\documentclass{amsart}

\usepackage{amssymb,amsmath,tikz}
\usepackage{mathtools}
\usetikzlibrary{positioning}
\usetikzlibrary{calc}
\usetikzlibrary{decorations.markings}
\usepackage[bookmarks=true,%
    colorlinks=true,%
    linkcolor=blue,%
    citecolor=blue,%
    filecolor=blue,%
    menucolor=blue,%
    urlcolor=blue,%
    breaklinks=true]{hyperref}

\newcommand{\Z}{\mathbb{Z}}
\newcommand{\C}{\mathbb{C}}
\newcommand{\R}{\mathbb{R}}

\newcommand{\K}{\mathbb{K}}
\newcommand{\G}{\mathbb{G}}
\newcommand{\mB}{\mathsf{B}}
\newcommand{\mytr}[1]{{#1}^{\top}}
\newtheorem{theorem}{Theorem}
\newtheorem{lemma}{Lemma}

\theoremstyle{definition}

\tikzstyle{hvector}=[draw=blue!50,fill=blue!10,thick]

\title{The Alexander polynomial as a universal invariant}
\author{Rinat Kashaev}
\address{Section de math\'ematiques, Universit\'e de Gen\`eve,
2-4 rue du Li\`evre, 1211 Gen\`eve 4, Suisse\\}
\email{rinat.kashaev@unige.ch}
\date{July 21, 2020.}
\thanks{2020MSC: 57K16, 57K14, 57K10}
\thanks{Supported in part by the Swiss National Science Foundation, the subsidy no~$200020\_192081$.}
\begin{document}
\begin{abstract} Let $\mB_1$ be the polynomial ring $\C[a^{\pm1},b]$ with the structure of a complex Hopf algebra induced from its interpretation as the algebra of regular functions on the affine linear algebraic group  of complex invertible upper triangular 2-by-2 matrices of the form $\left(
\begin{smallmatrix}
 a&b\\0&1
\end{smallmatrix}\right)$. We prove that the universal invariant of a long knot $K$ associated to $\mB_1$ is the reciprocal of the canonically normalised Alexander polynomial $\Delta_K(a)$. Given the fact that $\mB_1$ admits a $q$-deformation $\mB_q$ which underlies the (coloured) Jones polynomials, our result provides another conceptual interpretation for the Melvin--Morton--Rozansky conjecture proven by Bar-Nathan and Garoufalidis, and Garoufalidis and L\^e.
\end{abstract}
\maketitle
\section{Introduction}

The universal quantum knot invariants introduced and studied in a number of works~\cite{MR1025161,MR1124415,MR1153694,MR1227011,MR1324033, MR2186115,MR2253443,MR2251160} is a convenient tool allowing to encode the multitude of quantum invariants associated to a given Hopf algebra into a single algebraic object in completely representation independent way. As a result, the universal invariants are of great potential for conceptual understanding and organisation of the diversity of quantum invariants though the increased computational complexity makes them perhaps less useful for practical calculations. Nonetheless, as shows the example of the logarithmic invariants of Murakami--Nagatomo~\cite{MR2466562}, invariants associated to non semi-simple representations are sometimes better accessible through the universal invariants than directly from the R-matrix calculations.  

In this paper we address the problem of identification of the universal invariant of long knots in one of the simplest cases of non-trivial Hopf algebras, namely the  commutative but non co-commutative complex Hopf algebra $\mB_1:=\C[a^{\pm1},b]$ with the group-like element $a$ and the element $b$ with the co-product 
 \begin{equation}\label{eq:coprod-b}
\Delta(b)=a\otimes b+b\otimes 1.
\end{equation}
More abstractly, one can think of $\mB_1$ as the algebra of regular functions $\C[\operatorname{Aff}_1(\C)]$ on the affine linear algebraic group $\operatorname{Aff}_1(\C):=\G_a(\C)\rtimes \G_m(\C)$ of invertible upper triangular complex 2-by-2 matrices of the form $\left(
\begin{smallmatrix}
 a&b\\0&1
\end{smallmatrix}\right)$ where the Hopf algebra structure is canonically induced by the group structure of $\operatorname{Aff}_1(\C)$, see~\cite{MR547117}. 

The (maximal) Drinfeld's quantum double of $\mB_1$ is a Hopf algebra $D(\mB_1)$ which contains two Hopf sub-algebras: $\mB_1$ and the universal enveloping algebra of the Lie algebra of $\operatorname{Aff}_1(\C)$ generated by two primitive elements $\phi$ and $\psi$ subject to the commutation relation
\begin{equation}\label{eq:cr-lie-alg}
 \phi\psi-\psi\phi=\phi.
\end{equation}
Within the algebra $D(\mB_1)$, the element $a$ is  central while the element $b$ interacts with $\phi$ and $\psi$ through the commutation relations
\begin{equation}
 \phi b-b\phi=1-a ,\quad\psi b-b\psi=b.
\end{equation}
The formal universal R-matrix of $D(\mB_1)$ 
\begin{equation}\label{eq:un-r-mat}
 R=(1\otimes a)^{\psi\otimes 1}e^{\phi\otimes b}=\sum_{m,n\ge0}\frac 1{n!}\binom{\psi}{m}\phi^n\otimes (a-1)^m b^n,
\end{equation}
finds itself behind the associated universal invariant $Z_{\mB_1}(K)$ of a long knot $K$ which is a central element  of a specific ``profinite completion'' of $D(\mB_1)$ obtained as the convolution algebra $(D(\mB_1)^o)^*$ of the co-algebra structure of the restricted dual Hopf algebra $D(\mB_1)^o$. The following main result of this paper was conjectured in~\cite{Kashaev2019}.

\begin{theorem}\label{thm-1}
 The universal invariant of a long knot $K$ associated to the Hopf algebra $\mB_1$ is of the form
$
Z_{\mB_1}(K)=(\Delta_K(a))^{-1}
$
 where $\Delta_K(t)$ is the (canonically normalised) Alexander polynomial of $K$. %(normalised so that $\Delta_K(1)=1$ and $\Delta_K(t)=\Delta_K(1/t)$).
\end{theorem}

The reciprocal of the Alexander polynomial in this theorem should be thought of as an element of the ring of formal power series $\C[[a-1]]$ which naturally arises upon consideration of all finite dimensional representations of $D(\mB_1)$ where the element $a$  is always unipotent  (that is $a-1$ is nilpotent).
 Taking into account the close relationship of the Alexander polynomial with the Burau representation of the braid groups~\cite{MR3069652}, it is interesting to note that Salter in a recent work \cite{Salter2019} considers the Burau representation over the ring $\Z[[t-1]]$ in order to show that the Burau images of the braid groups are dense in Squier's unitary groups relative to the topology induced by the formal power series. 

The algebra $\mB_1$ can be $q$-deformed to a non-commutative Hopf algebra $\mB_q$ with the same co-algebra structure~\eqref{eq:coprod-b} but with a $q$-commutative relation $ab=qba$. For $q$ not a root of unity, the quantum double $D(\mB_q)$ contains the quantum group $U_q(sl_2)$ as a Hopf sub-algebra. In particular, for each $n\in\Z_{>0}$, it admits an $n$-dimensional irreducible representation corresponding to the $n$-th coloured Jones polynomial. In the large $n$ limit with $q=t^{1/n}$ and fixed $t$, one recovers an infinite-dimensional representation of the Hopf algebra $D(\mB_1)$ where the central element $a$ is realised by the scalar $t$. From that standpoint, Theorem~\ref{thm-1} is consistent with the Melvin--Morton--Rozansky conjecture proven by Bar-Nathan and Garoufalidis in \cite{MR1389962} and  by Garoufalidis and L\^e in \cite{MR2860990}.

Theorem~\ref{thm-1}, in conjunction with the group-like nature of the element $a$, makes absolutely transparent a result of Burau~\cite{MR3069587,MR3069631} on the property of the Alexander polynomial related to cables: if one takes the $n$-th cable of a knot $K$ and composes it with the braid of $n$ strands which brings the first strand under all others to $n$-th position, then the standard closure of the obtained string link gives a knot whose Alexander polynomial is $\Delta_K(t^n)$.

The main tool of the proof of Theorem~\ref{thm-1} is the use of a specific infinite dimensional representation of $D(\mB_1)$ on a dense vector subspace $A^1$ of a complex Hilbert space of square integrable holomorphic functions on $\C$ considered over the algebra of formal power series $\C[[\hbar]]$.  The evaluation of the formal universal R-matrix~\eqref{eq:un-r-mat} under this representation is a well defined element of the algebra $(\operatorname{End} (A^1))^{\otimes2}[[\hbar]]$, and it is this property of the R-matrix which, from one hand side, makes the corresponding Reshetikhin--Turaev functor well defined as formal power series in $\hbar$ despite the infinite dimensionality of the representation, and from the other hand, it allows us to use the Gaussian integrals to express the quantum invariant in terms of a minor of the unreduced Burau representation matrix.  Similarly to the work~\cite{MR1612375},  the identification of the quantum invariant with the Alexander polynomial is accomplished through a direct relationship of the latter to a minor of the unreduced Burau representation matrix.
\begin{theorem}[\cite{MR1133269}]\label{thm-2}
 Let a knot $K$ be the closure of a braid $\beta\in B_n$ and $\psi_n(\beta)\in \operatorname{GL}_n(\Z[t^{\pm1}])$ the image of $\beta$ under the unrestricted Burau representation (where the images of the standard Artin generators are linear in $t$). Let $\hat{\beta}_n$ be the $(n-1)\times(n-1)$ matrix obtained from $\psi_n(\beta)$ by throwing away the $n$-th column and the $n$-th row. Then, the Alexander polynomial of $K$ is given by the  formula
\begin{equation}\label{eq:re-alex-det-bur}
 \Delta_K(t)=t^{\frac{1-n-g(\beta)}2}\det(I_{n-1}-\hat\beta_n)
\end{equation}
where $I_k$ denotes the identity $k\times k$ matrix and $g\colon B_n\to \Z$ is the group homomorphism that sends the Artin generators  to 1.
\end{theorem}
Notice that the exponent of $t$ in the front factor of~\eqref{eq:re-alex-det-bur} is always an integer due to a specific parity property of the number $g(\beta)$. A proof of Theorem~\ref{thm-2}, based on the Alexander--Conway skein relation, is outlined in~\cite{MR1133269}. As an independent proof, we directly relate \eqref{eq:re-alex-det-bur}  to another known determinantal formula for the Alexander polynomial~\cite{MR0375281,MR2435235} that uses the reduced Burau representation and where the correcting multiplicative factor is slightly more complicated.
  
  \subsection*{Outline} Section~\ref{sec:ifha} starts with a concise review of the definition of the universal invariants from~\cite{Kashaev2019}, and then describes the center of $D(\mB_1)$. Remark that the universal invariant takes its values in a certain ``profinite completion'' of this center. Section~\ref{sec-schrodinger} introduces few algebraic and analytic tools used in the subsequent sections: the Hilbert spaces $H^n$ of holomorphic functions on $\C^n$ together with a particular class of elements called Schr\"odinger's coherent states (which are just linear exponential functions), the dense subspaces $A^n\subset H^n$ generated by products of coherent states and polynomials, and the Gaussian integration formula. The latter is the standard tool in quantum field theory which can also be thought of as an analytic version of MacMahon's Master theorem~\cite{MR0141605}. In Section~\ref{sec-repr} the representation of $D(\mB_1)$ in the space $A^1[[\hbar]]$ is introduced, the evaluation of the formal universal R-matrix is shown to be well defined and related to the basic building $2\times2$ matrix of the Burau representation of the braid groups, and the diagrammatic rules for calculation of the Reshetikhin--Turaev functor are specified. The last two Sections~\ref{sec-thm1} and \ref{sec-thm2} contain the proofs of Theorems~\ref{thm-1} and \ref{thm-2} respectively. 
  
  \subsection*{Acknowledgements} I would like to thank Louis-Hadrien Robert and Roland van der Veen for useful discussions.
  This work is partially supported by the Swiss National Science Foundation, the subsidy no~$200020\_192081$.
  
\section{Universal invariants of long knots from Hopf algebras}\label{sec:ifha}
In this section, based on the construction of R-matrix invariants of long knots in~\cite{MR1036112,MR1025161,MR2796628}, we briefly describe the definition of the universal invariants of long knots given in~\cite{Kashaev2019}, see also \cite{MR1324033,MR2251160} for an approach through the co-end. 

Consider the category $\mathbf{Hopf}_\K$ of Hopf algebras over a field $\K$ with invertible antipode. The restricted dual of an algebra provides us with a contravariant endofunctor $(\cdot)^o\colon \mathbf{Hopf}_\K\to \mathbf{Hopf}_\K$
which associates to a Hopf algebra $H$ with multiplication $\nabla$ the Hopf algebra 
\begin{equation}
H^o:=(\nabla^*)^{-1}(H^*\otimes H^*)\subset H^*
\end{equation}
whose underlying vector space is the vector subspace of the algebraic dual $H^*$ generated by all matrix coefficients of all finite dimensional representations of $H$~\cite{MR1786197}.

Drinfeld's quantum double  of $H\in \operatorname{Ob}\mathbf{Hopf}_\K$ (see, for example, \cite{MR1381692})  is a Hopf algebra $D(H)\in \operatorname{Ob}\mathbf{Hopf}_\K$ uniquely determined by the property that there are two Hopf algebra inclusions
\begin{equation}
\imath\colon H\to D(H),\quad \jmath\colon H^{o,\text{op}}\to D(H)
\end{equation}
such that $D(H)$ is generated by their images subject to the commutation relations
\begin{equation}\label{eq:comm-rel-j-i}
\jmath(f)\imath(x)= \langle f_{(1)},x_{(1)}\rangle  \langle f_{(3)},S(x_{(3)})\rangle\imath(x_{(2)})\jmath(f_{(2)})\quad \forall (x,f)\in H\times H^o
\end{equation}
where we use Sweedler's notation for the co-multiplication
\begin{equation}
\Delta(x)=x_{(1)}\otimes x_{(2)},\quad (\Delta\otimes\operatorname{id})(\Delta(x))=x_{(1)}\otimes x_{(2)}\otimes x_{(3)},\ \dots
\end{equation}

The restricted dual of the quantum double $D(H)^o$ is a  dual quasi-triangular Hopf algebra  with the  dual universal  R-matrix
\begin{equation}
\varrho\colon D(H)^o\otimes D(H)^o\to \K,\quad x\otimes y\mapsto \langle x,\jmath(\imath^o(y))\rangle
\end{equation}
which, among other things, satisfies the Yang--Baxter relation
\begin{equation}
\varrho_{1,2}*\varrho_{1,3}*\varrho_{2,3}=\varrho_{2,3}*\varrho_{1,3}*\varrho_{1,2}
\end{equation}
 in the convolution algebra $((D(H)^o)^{\otimes3})^*$. If $\{e_i\}_{i\in I}$ is a linear basis of $H$ and $\{e^i\}_{i\in I}$ is the associated set of canonical (dual) linear forms, then one can write a formal universal R-matrix  
 \begin{equation}\label{eq.un-R-mat}
R:=\sum_{i\in I}\jmath(e^i)\otimes\imath(e_i)
\end{equation}
as the formal conjugate of the dual universal  R-matrix in the sense of the equality
\begin{equation}
\langle x\otimes y,R\rangle=\langle \varrho, x\otimes y\rangle\quad\forall x,y\in D(H)^o.
\end{equation}
Furthermore, for any finite-dimensional right co-module
\begin{equation}
V\to V\otimes D(H)^o,\quad v\mapsto v_{(0)}\otimes v_{(1)},
\end{equation}
the  dual universal  R-matrix gives rise to a rigid R-matrix 
\begin{equation}
r_V\colon V\otimes V\to V\otimes V,\quad u\otimes v\mapsto v_{(0)}\otimes u_{(0)}\langle \varrho,v_{(1)}\otimes u_{(1)}\rangle.
\end{equation}
This implies that there exists a \emph{\color{blue} universal invariant} of long knots $Z_H(K)$ taking its values in the 
center of the convolution algebra $(D(H)^o)^*$ such that 
\begin{equation}
J_{r_V}(K)v=v_{(0)}\langle Z_H(K),v_{(1)}\rangle\quad \forall v\in V
\end{equation}
where $J_{r_V}(K)\in\operatorname{End}(V)$ is the invariant of long knots associated to $r_V$, see \cite{Kashaev2019} for details.

\subsection{The Hopf algebra $D(\mB_1)$ and its center}

 Recall that $\mB_1$ is the polynomial algebra $\C[a^{\pm1},b]$ provided with the structure of a Hopf algebra where $a$ is a  group-like element and the co-product of $b$ is given in~\eqref{eq:coprod-b}.

The opposite $\mB_1^{o,\text{op}}$  of the restricted dual Hopf algebra $\mB_1^o$ is composed of two Hopf sub-algebras: the group algebra $\C[\operatorname{Aff}_1(\C)]$ generated by group-like elements
\begin{equation}
\chi_{u,v},\quad (u,v)\in \C\times\C_{\ne0},\quad \chi_{u,v}\chi_{u',v'}=\chi_{u+vu',vv'},
\end{equation}
and the universal enveloping algebra $U(\operatorname{Lie}\operatorname{Aff}_1(\C))$ generated by
two primitive elements $\psi$ and $\phi$ satisfying the relation~\eqref{eq:cr-lie-alg}.
The relations between the generators of  $\C[\operatorname{Aff}_1(\C)]$ and  $U(\operatorname{Lie}\operatorname{Aff}_1(\C))$ are of the form
 \begin{equation}
 [\chi_{u,v},\psi]=u\phi \chi_{u,v},\quad \chi_{u,v} \phi =v\phi\chi_{u,v}\quad \forall (u,v)\in \C\times\C_{\ne0}
 \end{equation}
where $[x,y]:=xy-yx$.
As linear forms on $\mB_1$, they are defined by the relations
\begin{multline}
\langle\chi_{u,v},b^ma^n\rangle=u^mv^{-m-n},\\
\langle\phi,b^ma^n\rangle=\delta_{m,1},
\quad  \langle\psi,b^ma^n\rangle=\delta_{m,0}n,\quad \forall(m,n)\in\Z_{\ge0}\times \Z.
\end{multline}

The commutation relations~\eqref{eq:comm-rel-j-i} in the case of the quantum double $D(\mB_1)$ take the form
\begin{equation}
[\psi,b]=b,\quad
[\phi,b]=1-a,\quad b\chi_{u,v} =\chi_{u,v}(bv+(a-1)u)\quad \forall (u,v)\in \C\times\C_{\ne0}
\end{equation}
and $a$ is central. The formal universal R-matrix is given by formula~\eqref{eq:un-r-mat}.

Any finite dimensional right co-module $V$ over $D(\mB_1)^o$ is canonically a left module over $D(\mB_1)$
defined by
\begin{equation}
xw=w_{(0)}\langle w_{(1)},x\rangle,\quad \forall (x,w)\in D(\mB_1)\times V,
\end{equation}
 where the elements $a-1$, $b$ and $\phi$ are necessarily  nilpotent, so that the formal infinite double sum in \eqref{eq:un-r-mat}
truncates to a well defined finite sum.
\begin{lemma}
 The center of the algebra $D(\mB_1)$ is the polynomial sub-algebra $\C[a^{\pm1},c]$ where
\begin{equation}\label{eq:cent-el-c}
c:=\phi b+(a-1)\psi.
\end{equation}
\end{lemma}
\begin{proof}
 It is easily verified that $c$ is central. Any element $x\in D(\mB_1)$ can  uniquely be written in the form
\begin{equation}
x=\sum_{(u,v,m)\in \C\times \C_{\ne0}\times \Z}\chi_{u,v} e_m p_{u,v,m}(a,c,\psi),
\end{equation}
where
 \begin{equation}
e_m:=\left\{
\begin{array}{cc}
 b^m &  \text{if }\  m>0;   \\
 1 &  \text{if }\  m=0;    \\
 \phi^{-m} & \text{if }\  m<0
\end{array}
\right.
\end{equation}
and $ p_{u,v,m}(a,c,\psi)\in\C[a^{\pm1},c,\psi]$ is non-zero for only finitely many triples $(u,v,m)$.

Assume that $x$ is central.  Then, for any $s\in\C_{\ne0}$, we have the equality
\begin{multline}
x=\chi_{0,s}^{-1} x \chi_{0,s}=\sum_{(u,v,m)\in \C\times \C_{\ne0}\times \Z}\chi_{u/s,v} e_m s^m p_{u,v,m}(a,c,\psi)\\
=\sum_{(u,v,m)\in \C\times \C_{\ne0}\times \Z}\chi_{u,v} e_m s^m p_{us,v,m}(a,c,\psi)
\end{multline}
which implies that for any fixed triple $(u,v,m)\in  \C\times \C_{\ne0}\times \Z$, one has the family of equalities
\begin{equation}
p_{u,v,m}=s^m p_{us,v,m}\quad \forall s\in\C_{\ne0}.
\end{equation}
This means that $p_{u,v,m}$ can only be non-zero if $u=m=0$. Thus, the element $x$ takes the form
\begin{equation}
x=\sum_{v\in\C_{\ne0}}\chi_{0,v} p_{0,v,0}(a,c,\psi).
\end{equation}
The equality
 \begin{equation}
bx=xb=b\sum_{v\in\C_{\ne0}}\chi_{0,v} v^{-1}p_{0,v,0}(a,c,\psi+1).
\end{equation}
is equivalent to the equalities
\begin{equation}
 p_{0,v,0}(a,c,\psi+1)=v^{-1}p_{0,v,0}(a,c,\psi) \quad \forall v\in \C_{\ne0}
\end{equation}
which imply that the polynomial $p_{0,v,0}(a,c,\psi)$ can be non-zero only if $v=1$ and if it does not depend on $\psi$. We conclude that $x\in \C[a^{\pm1},c]$.
\end{proof}
\section{Schr\"odinger's coherent states} \label{sec-schrodinger}
Here we briefly review the theory of standard Schr\"odinger's coherent states (see for example \cite{MR858831}).  

For any $n\in\Z_{>0}$, let  $H^n\subset L^2(\C^n,\mu_n)$ be the complex Hilbert space of square integrable holomorphic functions $f\colon \C^n \to \C$ with the scalar product
\begin{equation}
\langle f\vert g\rangle :=\int_{\C^n}\overline{f(z)} g(z)\operatorname{d}\!\mu_n(z)
\end{equation}
where the measure $\mu_n$ on $\C^n$ is absolutely continuous with respect to the Lebesgue measure $\lambda_{2n}$ on $\C^n\simeq \R^{2n}$ with the Radon--Nikodym derivative
\begin{equation}
\frac{\operatorname{d}\!\mu_n}{\operatorname{d}\!\lambda_{2n}}(z)=\frac1{\pi^n} e^{-\|z\|^2},\quad \|z\|:=\sqrt{\sum\nolimits_{i=0}\nolimits^{n-1}|z_i|^2}.
\end{equation}
By direct calculation, one verifies that the monomials
\begin{equation}
e_k(z):=\prod_{i=0}^{n-1}\frac{z_i^{k_i}}{\sqrt{k_i!}},\quad k\in \Z_{\ge0}^n
\end{equation}
form an orthonormal family in $H^n$ which is, in fact, a Hilbert basis due to the validity of  Taylor's  (multivariable) expansion for holomorphic functions:
\begin{equation}\label{eq:taylor-exp}
f(z)=\sum_{k\in \Z_{\ge0}^n}\prod_{i=0}^{n-1}\frac{z_i^{k_i}}{k_i!}\frac{\partial^{k_i}f(w)}{\partial w_i^{k_i}}\biggr\vert_{w=0}=\sum_{k\in \Z_{\ge0}^n}e_k(z)\prod_{i=0}^{n-1}\frac{1}{\sqrt{k_i!}}\frac{\partial^{k_i}f(w)}{\partial w_i^{k_i}}\biggr\vert_{w=0}
\end{equation}
which, in the case when $f\in H^n$, implies that 
\begin{equation}\label{eq:fourier-coeff-hol}
\int_{\C^n}\overline {e_k(z)}f(z)\operatorname{d}\!\mu_n(z)=\prod_{i=0}^{n-1}\frac{1}{\sqrt{k_i!}}\frac{\partial^{k_i}f(w)}{\partial w_i^{k_i}}\biggr\vert_{w=0}\quad \forall k\in \Z_{\ge0}^n.
\end{equation}
For any $u\in\C^n$, multiplying both sides of~\eqref{eq:fourier-coeff-hol} by $e_k(u)$, summing over all $k\in  \Z_{\ge0}^n$ and, using the Fubini (or dominant convergence) theorem in the left hand side for exchanging the integration and summation,  and the Taylor formula~\eqref{eq:taylor-exp} in the right hand side, we obtain
\begin{equation}\label{eq:rep-prop-cs}
\int_{\C^n}\varphi_u(\bar z)f(z)\operatorname{d}\!\mu_n(z)=f(u)\quad \forall f\in H^n
\end{equation}
where the holomorphic function
\begin{equation}
\varphi_u\colon \C^n\to\C,\quad z\mapsto \sum_{k\in  \Z_{\ge0}^n}e_k(u)e_k(z)=e^{\sum_{i=0}^{n-1}u_i z_i}
\end{equation}
determines an element
 $
\varphi_u\in H^n%, \quad \varphi_\alpha(z)=e^{\mytr{\alpha}z}
$
called  \emph{\color{blue} (Schr\"odinger's) coherent state}.
By treating elements of $\C^n$ as column vectors we can write $\varphi_u(z)=e^{\mytr{u}z}$. Let us also adopt the notation $w^*:=\mytr{\bar w}$ for the Hermitian conjugation, i.e. the transposition combined with the complex conjugation. With this notation we have the equalities
\begin{equation}
\|w\|^2=w^*w,\quad \varphi_u(\bar z)=e^{z^*u}.
\end{equation}

The integral formula~\eqref{eq:rep-prop-cs} expresses the reproducing property of coherent states
\begin{equation}
\langle \varphi_{\bar u}\vert f\rangle=f(u)\quad \forall (f,u)\in H^n\times\C^n.
\end{equation}
The choice $f=\varphi_v$ in the last formula gives the scalar product between the coherent states
\begin{equation}\label{eq:sc-pr-coh-st}
 \langle \varphi_{\bar u} | \varphi_v\rangle=\varphi_{v}(u)=\varphi_u(v)=e^{\mytr{u}v}.
\end{equation}
In particular, the norm of a coherent state $\varphi_v$ is determined by the Euclidean norm  of $v$ through the formula
\begin{equation}
\|\varphi_v\|=e^{\|v\|^2/2}.
\end{equation}

\subsection{A dense subspace of $H^n$}
Another useful property of the coherent states is that the (dense) vector subspace  $A^n$ of $H^n$ generated by products of coherent states and polynomials  is stable under the multiplication of elements of $A^n$ as functions so that $A^n$ carries the additional structure of a commutative algebra, and it is in the domain of any linear differential operator with coefficients in $A^n$.
For example, when $n=1$, the Hilbert basis of $H^1$ given by the monomials $\{e_k\}_{k\in\Z_{\ge0}}\subset A^1$ is the eigenvector basis of the 1-dimensional quantum harmonic oscillator with the (self-adjoint) Hamiltonian operator $z\frac{\partial}{\partial z}$. 

\subsection{Gaussian integration formula}
Writing out explicitly the scalar product as an integral in~\eqref{eq:sc-pr-coh-st}, we obtain an integral identity
\begin{equation}
\int_{\C^n} e^{\mytr{v}z+z^*u}\operatorname{d}\!\mu_n(z)=e^{\mytr{v}u}
\end{equation}
which is a special case of the general Gaussian integration formula
\begin{equation}\label{eq:gauss-int-form}
\int_{\C^n} e^{v^*z+z^*u + z^*Mz}\operatorname{d}\!\mu_n(z)=\frac{e^{v^*W^{-1}u}}{\det(W)},\quad W:=I_n-M
\end{equation}
where $M$ is an arbitrary complex $n$-by-$n$ matrix sufficiently close to zero so that the integral is absolutely convergent. Furthermore, the expansion  of~\eqref{eq:gauss-int-form} in power series in $M$ with $u=v=0$ corresponds to the purely combinatorial MacMahon Master theorem~\cite{MR0141605}.

%Additionally, by using the standard action of the affine group $\operatorname{Aff}_n(\C):=\C^n\rtimes\operatorname{GL}_n(\C)$  on the space $\C^n$, $(v,g)z=v+gz$,  the space $A^n$ is an invariant subspace of the space of analytic functions under the representation of $\operatorname{Aff}_n(\C)$ \begin{equation}\label{eq:rep-aff-gr}\rho^{(n)}(v,g)f(z):=f((v,g)^{-1}z)=f(g^{-1}(z-v))\quad \forall (v,g)\in\operatorname{Aff}_n(\C)\end{equation}which acts on the coherent states by the formula\begin{equation}\rho^{(n)}(v,g)\varphi_u(z)=\varphi_{\mytr{(g^{-1})}u}(-v)\varphi_{\mytr{(g^{-1})}u}(z).\end{equation}

\section{Representations of $D(\mB_1)$ in $A^1[[\hbar]]$}\label{sec-repr}

Recall that $A^1$ is the vector subspace of $H^1$ generated by products of coherent states with polynomials. For any $\lambda\in \C$, the mappings 
 \begin{equation}\label{eq:rho-hom}
a\mapsto 1+\hbar,\quad b\mapsto \frac{\partial}{\partial z},\quad \phi\mapsto \hbar z,\quad \psi\mapsto \lambda-z\frac{\partial}{\partial z}
\end{equation}
and the action
\begin{equation}
\chi_{u,v}f(z)=e^{\hbar uz}f(vz)
\end{equation}
determine a homomorphism of algebras 
\begin{equation}\label{eq:rst-rep}
\rho_\lambda\colon D(\mB_1)\to \operatorname{End}(A^1[[\hbar]])
\end{equation}
 which sends the central element $c$ defined in \eqref{eq:cent-el-c} to $\lambda\hbar$. 

An important property of the representation $\rho_\lambda$ is that the image under $ \rho_\lambda^{\otimes 2}$ of the formal R-matrix~\eqref{eq:un-r-mat} is a well defined element of the algebra
$ \operatorname{End}(A^1)^{\otimes 2}[[\hbar]]$:

\begin{equation}\label{eq:eval-r-mat}
 \rho_\lambda^{\otimes 2}(R)=(1+\hbar)^{\lambda-z_0\frac{\partial}{\partial z_0}}e^{\hbar z_0\frac{\partial}{\partial z_1}}=\sum_{m,n\ge0}\frac{\hbar^{m+n}}{n!}\binom{\lambda-z_0\frac{\partial}{\partial z_0}}{m} \Big(z_0\frac{\partial}{\partial z_1}\Big)^n.
\end{equation}
In particular, the double sum in~\eqref{eq:eval-r-mat} truncates to a finite sum if the indeterminate $\hbar$ is nilpotent. 
Thus, despite the fact that the representation $\rho_\lambda$ is infinite dimensional, the corresponding R-matrix is well suited for calculation of the image under $\rho_\lambda$ of the universal invariant $Z_{\mB_1}(K)$. Moreover, as the parameter $\lambda$ enters only through the overall normalisation factor $(1+\hbar)^\lambda$ of the R-matrix, the associated invariant  is independent of $\lambda$. For that reason, in what follows, we put $\lambda=0$ and work only with the representation $\rho:=\rho_0$.

In order to apply the construction of~\cite{Kashaev2019}, we define the input R-matrix
\begin{equation}
 r:=\rho^{\otimes 2}(R)P
\end{equation}
where $P\in \operatorname{Aut}(A^2)$ is the permutation operator acting by exchanging the arguments.
 By using~\eqref{eq:eval-r-mat}, we obtain the following explicit action of $r$:
\begin{multline}
rf(z)= rf(z_0,z_1)=(1+\hbar)^{-z_0\frac{\partial}{\partial z_0}}f(z_1+\hbar z_0,z_0)\\
 =f\Big(z_1+\frac{\hbar}{1+\hbar}z_0,\frac{1}{1+\hbar}z_0\Big)
 =f(\mytr{U}z)%=\rho^{(2)}(0,U^{-1})f(z)
\end{multline}
where 
\begin{equation}\label{eq:U-matrix}
U :=\left(\begin{matrix}
 \frac{\hbar}{1+\hbar}&\frac1{1+\hbar}\\
 1&0
\end{matrix}\right)=\left(\begin{matrix}
 1-t&t\\
 1&0
\end{matrix}\right),\quad t:=\frac1{1+\hbar},
\end{equation}
is the 2-by-2 matrix entering the definition of the (unrestricted) Burau representation of the braid groups~\cite{MR3069652}.  %Thus,  the R-matrix $r$ induces a representation of the braid group $B_n$ on $A^n[[\hbar]]$ given by the formula $\rho^{(n)}(0,\hat\beta^{-1})$ where $\hat\beta\in\operatorname{GL}_n(\Z[t,t^{-1}])\subset \operatorname{GL}_n(\Z[[\hbar]])$ is the image of a braid $\beta$ of $n$ strands under the Burau representation.
The action of $r$ on the coherent states is realized by the action of the transposed matrix on the space of parameters: 
\begin{equation}\label{eq:r-acts-on-phi}
r\varphi_{v}(z)=\varphi_v(\mytr{U} z) =\varphi_{U v}(z).%,\quad \check\sigma:=\mytr{\hat \sigma}.
\end{equation}
In what follows, we use the indeterminate $t$ defined in terms of $\hbar$ through the formula in~\eqref{eq:U-matrix}.

\subsection{The diagrammatic rules for the Reshetikhin--Turaev functor}
From the formula~\eqref{eq:r-acts-on-phi}, one calculates the integral kernel of $r$ with respect to the coherent states
\begin{equation}
\langle \varphi_{w}|r|\varphi_{v}\rangle=\langle \varphi_{w_0,w_1}|r|\varphi_{v_0,v_1}\rangle=e^{w^*U v}
\end{equation}
which corresponds to the value of the Reshetikhin--Turaev functor associated to positive crossings of all orientations in normal long knot diagrams with edges coloured by complex numbers:
\begin{equation}
\begin{tikzpicture}[scale=1,baseline=10]
\draw[thick,<-] (0,1) to [out=-90,in=90] (1,0);
\draw[line width=3pt,white] (1,1) to [out=-90,in=90] (0,0);
\draw[thick,<-] (1,1) to [out=-90,in=90] (0,0);
\node (sw) at (0,-.1){\tiny $v_0$};\node (se) at (1,-.1){\tiny $v_1$};
\node (nw) at (0,1.1){\tiny $w_0$};\node (ne) at (1,1.1){\tiny $w_1$};
\end{tikzpicture},
\begin{tikzpicture}[scale=1,baseline=10]
\draw[thick,->] (1,1) to [out=-90,in=90] (0,0);
\draw[line width=3pt,white] (0,1) to [out=-90,in=90] (1,0);
\draw[thick,<-] (0,1) to [out=-90,in=90] (1,0);
\node (sw) at (0,-.1){\tiny $w_0$};\node (se) at (1,-.1){\tiny $v_0$};
\node (nw) at (0,1.1){\tiny $w_1$};\node (ne) at (1,1.1){\tiny $v_1$};
\end{tikzpicture},
\begin{tikzpicture}[scale=1,baseline=10]
\draw[thick,->] (0,1) to [out=-90,in=90] (1,0);
\draw[line width=3pt,white] (1,1) to [out=-90,in=90] (0,0);
\draw[thick,->] (1,1) to [out=-90,in=90] (0,0);
\node (sw) at (0,-.1){\tiny $w_1$};\node (se) at (1,-.1){\tiny $w_0$};
\node (nw) at (0,1.1){\tiny $v_1$};\node (ne) at (1,1.1){\tiny $v_0$};
\end{tikzpicture},
\begin{tikzpicture}[scale=1,baseline=10]
\draw[thick,<-] (1,1) to [out=-90,in=90] (0,0);
\draw[line width=3pt,white] (0,1) to [out=-90,in=90] (1,0);
\draw[thick,->] (0,1) to [out=-90,in=90] (1,0);
\node (sw) at (0,-.1){\tiny $v_1$};\node (se) at (1,-.1){\tiny $w_1$};
\node (nw) at (0,1.1){\tiny $v_0$};\node (ne) at (1,1.1){\tiny $w_0$};
\end{tikzpicture}\  \xmapsto{RT_r}\  \langle \varphi_{w_0,w_1}|r|\varphi_{v_0,v_1}\rangle
\end{equation}
Likewise, the integral kernel of $r^{-1}$ given by the formula
\begin{equation}
\langle \varphi_{w}|r^{-1}|\varphi_{v}\rangle=\langle \varphi_{w_0,w_1}|r^{-1}|\varphi_{v_0,v_1}\rangle=e^{w^*U^{-1} v}
\end{equation}
is associated to negative crossings of all orientations:
\begin{equation}
\begin{tikzpicture}[scale=1,baseline=10]
\draw[thick,<-] (1,1) to [out=-90,in=90] (0,0);
\draw[line width=3pt,white] (0,1) to [out=-90,in=90] (1,0);
\draw[thick,<-] (0,1) to [out=-90,in=90] (1,0);
\node (sw) at (0,-.1){\tiny $v_0$};\node (se) at (1,-.1){\tiny $v_1$};
\node (nw) at (0,1.1){\tiny $w_0$};\node (ne) at (1,1.1){\tiny $w_1$};
\end{tikzpicture},
\begin{tikzpicture}[scale=1,baseline=10]
\draw[thick,<-] (0,1) to [out=-90,in=90] (1,0);
\draw[line width=3pt,white] (1,1) to [out=-90,in=90] (0,0);
\draw[thick,->] (1,1) to [out=-90,in=90] (0,0);
\node (sw) at (0,-.1){\tiny $w_0$};\node (se) at (1,-.1){\tiny $v_0$};
\node (nw) at (0,1.1){\tiny $w_1$};\node (ne) at (1,1.1){\tiny $v_1$};
\end{tikzpicture},
\begin{tikzpicture}[scale=1,baseline=10]
\draw[thick,->] (1,1) to [out=-90,in=90] (0,0);
\draw[line width=3pt,white] (0,1) to [out=-90,in=90] (1,0);
\draw[thick,->] (0,1) to [out=-90,in=90] (1,0);
\node (sw) at (0,-.1){\tiny $w_1$};\node (se) at (1,-.1){\tiny $w_0$};
\node (nw) at (0,1.1){\tiny $v_1$};\node (ne) at (1,1.1){\tiny $v_0$};
\end{tikzpicture}, 
\begin{tikzpicture}[scale=1,baseline=10]
\draw[thick,->] (0,1) to [out=-90,in=90] (1,0);
\draw[line width=3pt,white] (1,1) to [out=-90,in=90] (0,0);
\draw[thick,<-] (1,1) to [out=-90,in=90] (0,0);
\node (sw) at (0,-.1){\tiny $v_1$};\node (se) at (1,-.1){\tiny $w_1$};
\node (nw) at (0,1.1){\tiny $v_0$};\node (ne) at (1,1.1){\tiny $w_0$};
\end{tikzpicture}\ \xmapsto{RT_r}\  \langle \varphi_{w_0,w_1}|r^{-1}|\varphi_{v_0,v_1}\rangle.
\end{equation}
We complete the list of the diagrammatic rules by adding the rules for vertical segments and  local extrema 
\begin{equation}
\begin{tikzpicture}[yscale=.5,baseline]
\draw[thick,->] (0,0) to [out=90,in=-90] (0,1);
\node (n) at (0,1.2){\tiny $w$};\node (s) at (0,-.2){\tiny $v$};
\end{tikzpicture}\ ,
\begin{tikzpicture}[yscale=.5,baseline]
\draw[thick,<-] (0,0) to [out=90,in=-90] (0,1);
\node (n) at (0,1.2){\tiny $v$};\node (s) at (0,-.2){\tiny $w$};
\end{tikzpicture}\ ,
\begin{tikzpicture}[xscale=1,baseline=0]
\draw[thick,<-] (0,0) to [out=90,in=90] (1,0);
\node (w) at (0,-.1){\tiny $w$};\node (e) at (1,-.1){\tiny $v$};
\end{tikzpicture}\ ,
\begin{tikzpicture}[xscale=1,baseline=25]
\draw[thick,<-] (0,1) to [out=-90,in=-90] (1,1);
\node (w) at (0,1.1){\tiny $w$};\node (e) at (1,1.1){\tiny $v$};
\end{tikzpicture}\ \xmapsto{RT_r} \ e^{\bar w v}
\end{equation}
where $e^{\bar w v}$ is the integral kernel of the  identity operator $\operatorname{id}_{A^1}$:
\begin{equation}
\langle \varphi_{w}|\operatorname{id}_{A^1}|\varphi_{v}\rangle=\langle \varphi_{w}|\varphi_{v}\rangle=e^{\bar w v}.
\end{equation}

For later use, we calculate the following two Reshetikhin--Turaev images
\begin{multline}\label{eq:+cup}
\langle\varphi_w|RT_r\Big(
\begin{tikzpicture}[yscale=1,baseline=18]
%\draw [fill=gray!25](.3,0) rectangle (1.7,1);
\coordinate (a1) at (0,1);
\coordinate (a2) at (.75,.5);
\coordinate (a3) at (0.25,.5);
\coordinate (a4) at (1,1);
%\node at (0,1.1) {\tiny$v$};
%\node at (1,1.1) {\tiny$w$};
\draw[thick] (a2) to [out=-90,in=-90] (a3);
\draw[thick,->] (a3) to [out=90,in=-135] (a4);
\draw[line width=3pt,white] (a1) to [out=-45,in=90] (a2);
\draw[thick] (a1) to [out=-45,in=90] (a2);
\end{tikzpicture}
\Big)|\varphi_v\rangle=
RT_r\Big(
\begin{tikzpicture}[yscale=1,baseline=18]
%\draw [fill=gray!25](.3,0) rectangle (1.7,1);
\coordinate (a1) at (0,1);
\coordinate (a2) at (.75,.5);
\coordinate (a3) at (0.25,.5);
\coordinate (a4) at (1,1);
\node at (0,1.1) {\tiny$v$};
\node at (1,1.1) {\tiny$w$};
\draw[thick] (a2) to [out=-90,in=-90] (a3);
\draw[thick,->] (a3) to [out=90,in=-135] (a4);
\draw[line width=3pt,white] (a1) to [out=-45,in=90] (a2);
\draw[thick] (a1) to [out=-45,in=90] (a2);
\end{tikzpicture}
\Big)
=\int_{\C}\langle\varphi_{w,u}|r|\varphi_{v,u}\rangle\operatorname{d}\!\mu_1(u)\\
=\int_{\C}e^{\left(\begin{smallmatrix}
 \bar w&\bar u 
\end{smallmatrix}\right)
\left(\begin{smallmatrix}
 1-t&t\\
 1&0
\end{smallmatrix}\right)
\left(\begin{smallmatrix}
 v\\
 u
\end{smallmatrix}\right)
}\operatorname{d}\!\mu_1(u)
=\int_{\C}e^{
\bar w(1-t)v+\bar w tu+\bar uv
}\operatorname{d}\!\mu_1(u)=e^{\bar wv}
\end{multline}
and 
\begin{multline}\label{eq:-cup}
\langle\varphi_w|RT_r\Big(
\begin{tikzpicture}[yscale=1,baseline=18]
%\draw [fill=gray!25](.3,0) rectangle (1.7,1);
\coordinate (a1) at (0,1);
\coordinate (a2) at (.75,.5);
\coordinate (a3) at (0.25,.5);
\coordinate (a4) at (1,1);
%\node at (0,1.1) {\tiny$v$};
%\node at (1,1.1) {\tiny$w$};
\draw[thick] (a1) to [out=-45,in=90] (a2);
\draw[thick] (a2) to [out=-90,in=-90] (a3);
\draw[line width=3pt,white] (a3) to [out=90,in=-135] (a4);
\draw[thick,->] (a3) to [out=90,in=-135] (a4);
\end{tikzpicture}
\Big)|\varphi_v\rangle=
RT_r\Big(
\begin{tikzpicture}[yscale=1,baseline=18]
%\draw [fill=gray!25](.3,0) rectangle (1.7,1);
\coordinate (a1) at (0,1);
\coordinate (a2) at (.75,.5);
\coordinate (a3) at (0.25,.5);
\coordinate (a4) at (1,1);
\node at (0,1.1) {\tiny$v$};
\node at (1,1.1) {\tiny$w$};
\draw[thick] (a1) to [out=-45,in=90] (a2);
\draw[thick] (a2) to [out=-90,in=-90] (a3);
\draw[line width=3pt,white] (a3) to [out=90,in=-135] (a4);
\draw[thick,->] (a3) to [out=90,in=-135] (a4);
\end{tikzpicture}
\Big)
=\int_{\C}\langle\varphi_{w,u}|r^{-1}|\varphi_{v,u}\rangle\operatorname{d}\!\mu_1(u)\\
=\int_{\C}e^{\left(\begin{smallmatrix}
 \bar w&\bar u 
\end{smallmatrix}\right)
\left(\begin{smallmatrix}
 0&1\\
 t^{-1}&1-t^{-1}
\end{smallmatrix}\right)
\left(\begin{smallmatrix}
 v\\
 u
\end{smallmatrix}\right)
}\operatorname{d}\!\mu_1(u)
=\int_{\C}e^{
\bar w u+\bar ut^{-1}v+\bar u(1-t^{-1})u
}\operatorname{d}\!\mu_1(u)=te^{\bar wv}
\end{multline}
where the integrals are calculated by using the Gaussian integration formula~\eqref{eq:gauss-int-form}.
\section{Proof of Theorem~\ref{thm-1}}\label{sec-thm1}
Let $K$ be represented by the closure of a braid $\beta\in B_n$.
Let us choose a normal long knot diagram $D_\beta$ representing $K$ according to the picture %in~\eqref{pic:long-knot}
\begin{equation}\label{pic:long-knot}
D_\beta=
\begin{tikzpicture}[baseline=-2]
%\node (x) {$f\otimes g$};
 \node (a)  [hvector] at (0,0) {$D_\beta$};
 \draw[thick,->] (a)--(0,0.5);
  \draw[thick] (a)--(0,-0.5);
\end{tikzpicture}
=\ 
\begin{tikzpicture}[yscale=.5,baseline]
\draw [%fill=gray!25
hvector](.3,0) rectangle (1.7,1);
\coordinate (a1) at (-.4,.5);
\coordinate (a2) at (.4,-.7);
\coordinate (a3) at (0.1,-.7);
\node (center) at (1,.5){ $\beta$};
\node (udots) at (.8,1.1){\tiny$\ldots$};
\node (bdots) at (.8,-.1){\tiny$\ldots$};
\node (cdots) at (-.7,.5){\tiny$\ldots$};
%\node (ddots) at (.25,-.95){\tiny$\vdots$};
\draw[thick,->] (.5,1) to [out=120,in=90] (a1);
\draw[thick] (a2) to [out=-90,in=-90] (a3);
\draw[thick] (a3) to [out=90,in=-120] (0.5,0);
\draw[line width=3pt,white]  (a1) to [out=-90,in=90] (a2);
\draw[thick] (a1) to [out=-90,in=90] (a2);

\coordinate (b1) at (-1,.5);
\coordinate (b2) at (.6,-2);
\coordinate (b3) at (0.3,-2);
\draw[thick,->] (1.3,1) to [out=120,in=90] (b1);
\draw[thick] (b2) to [out=-90,in=-90] (b3);
\draw[thick] (b3) to [out=90,in=-120] (1.3,0);
\draw[line width=3pt,white] (b1) to [out=-90,in=90] (b2);
\draw[thick] (b1) to [out=-90,in=90] (b2);

\draw[thick,->] (1.5,1) to [out=90,in=-90] (1.5,1.5);
\draw[thick] (1.5,-2) to [out=90,in=-90] (1.5,0);
%\draw[line width=3pt,white] (0,1) to [out=-90,in=90] (1,0);
%\draw[thick,<-] (0,1) to [out=-90,in=90] (1,0);

%\node (nw) at (0,1.1){\tiny $w_0$};\node (ne) at (1,1.1){\tiny $w_1$};
\end{tikzpicture}
\end{equation}
with the writhe $g(D_\beta)=g(\beta)+n-1$ which is an even number.
Taking into account the value~\eqref{eq:+cup},
writing the matrix $\psi_n(\beta)$ in the block form
\begin{equation}
\psi_n(\beta)=
\begin{pmatrix}
 \hat\beta_n&b_\beta\\
 c_\beta& d_\beta
\end{pmatrix},
\end{equation}
and using the general Gaussian integration formula~\eqref{eq:gauss-int-form}, we calculate
\begin{multline}\label{eq:rt_r-d_beta}
\langle \varphi_w| RT_r(D_\beta)| \varphi_v\rangle=RT_r\bigg(\begin{tikzpicture}[baseline=-3]
%\node (x) {$f\otimes g$};
 \node (a)  [hvector] at (0,0) {$D_\beta$};
 \node (up) [] at (0,0.65){\tiny$w$};
  \node (do) [] at (0,-0.6){\tiny$v$};
 \draw[thick,->] (a)--(up);
  \draw[thick] (a)--(do);
\end{tikzpicture}\bigg)
=\int_{\C^{n-1}}e^{\left(\begin{smallmatrix}
 u^*&\bar w 
\end{smallmatrix}\right)\psi_n(\beta)\left(\begin{smallmatrix}
 u\\
 v 
\end{smallmatrix}\right)}\operatorname{d}\!\mu_{n-1}(u)\\
=\int_{\C^{n-1}}e^{\left(\begin{smallmatrix}
 u^*&\bar w 
\end{smallmatrix}\right)\left(\begin{smallmatrix}
 \hat\beta_n&b_\beta\\
 c_\beta& d_\beta
\end{smallmatrix}\right)\left(\begin{smallmatrix}
 u\\
 v 
\end{smallmatrix}\right)}\operatorname{d}\!\mu_{n-1}(u)\\
=
\int_{\C^{n-1}}e^{\bar wd_\beta v+\bar wc_\beta u+u^*b_\beta v+u^*\hat\beta_n u}\operatorname{d}\!\mu_{n-1}(u)
=\frac{e^{\bar wd_\beta v+\bar wc_\beta(I_{n-1}-\hat\beta_n)^{-1}b_\beta v}}{\det(I_{n-1}-\hat\beta_n)}.
\end{multline}
On the other hand, given the fact that we are calculating a central element realised by a scalar so that on \`a priori grounds the result should be proportional to the integral kernel of the identity operator $e^{\bar w v}$, we conclude that the identity
\begin{equation}
d_\beta+c_\beta(I_{n-1}-\hat\beta_n)^{-1}b_\beta=1
\end{equation}
is satisfied, a property of $\psi_n(\beta)$ which does not look to be easy to prove without passing through the Gaussian integration and referring to the universal invariant.

Finally, it remains to take into account the writhe correction, which, according to the values in~\eqref{eq:+cup} and \eqref{eq:-cup} is given by the formula
\begin{equation}\label{eq:writhe-cor}
\langle \varphi_w| RT_r\big(\xi^{-g(D_\beta)/2}\big)| \varphi_v\rangle e^{-\bar w v} =t^{g(D_\beta)/2}=t^{(g(\beta)+n-1)/2}
\end{equation}
where we use the notation $\xi^k$ from \cite{Kashaev2019} for a specific class of long knot diagrams used to compensate the writhe of the  diagram.
Putting together \eqref{eq:rt_r-d_beta} and \eqref{eq:writhe-cor}, the result for the invariant  $J_r(K)$ reads
\begin{equation}
\langle \varphi_w|J_r(K)| \varphi_v\rangle e^{-\bar w v}=\frac{t^{(g(\beta)+n-1)/2} }{\det(I_{n-1}-\hat\beta_n)}=\frac1{\Delta_K(t)}
\end{equation}
where the last equality is due to formula~\eqref{eq:re-alex-det-bur}. Taking into account the relation between $\hbar$, $t$ and the realisation of the central element $a$ of $D(\mB_1)$ as well as the symmetry of the Alexander polynomial under the substitution $t\mapsto t^{-1}$, we conclude the proof.

\section{Proof of Theorem~\ref{thm-2}}\label{sec-thm2}
In this section, we adopt the notation  of~\cite{MR2435235} and first briefly describe the unreduced and reduced Burau representations of the braid groups $B_n$ for $n\ge2$. 

For any $k\ge 1$, denote by $I_k$ the identity $k\times k$ matrix.  Let 
\begin{equation}
\psi_n\colon B_n\to\operatorname{GL}_n(\Lambda),\quad \Lambda:=\Z[t^{\pm1}],
\end{equation}
be the unrestricted Burau representation where Artin's standard generators $\sigma_i$, $1\le i< n$ , are realised by the matrices
\begin{equation}
\psi_n(\sigma_i)= U_i:= I_{i-1}\oplus U\oplus  I_{n-i-1}.
\end{equation}
For any $k\ge 1$, define the invertible upper triangular $k\times k$ matrix  
\begin{equation}
C_k=\sum_{1\le i\le j\le k}E_{i,j}=I_k+\sum_{1\le i< j\le k}E_{i,j}
\end{equation}
where $E_{i,j}$ is the matrix with the only non-zero element 1 at the place  $(i,j)$. Its inverse has the form
\begin{equation}
C_k^{-1}=I_k-\sum_{i=1}^{n-1}E_{i,i+1}.
\end{equation}
Indeed, one easily calculates
\begin{equation}
C_k(I_k-\sum_{i=1}^{n-1}E_{i,i+1})=C_k-\sum_{1\le i< j\le k}E_{i,j}=I_k.
\end{equation}
We remark on the block structure of $C_k^{\pm1}$:
\begin{equation}\label{eq:block-c_n}
C_k=\begin{pmatrix}
C_{k-1}&1_{k-1}\\
0_{k-1}^\top &1
\end{pmatrix},\quad C_k^{-1}=\begin{pmatrix}
C_{k-1}^{-1}&-C_{k-1}^{-1}1_{k-1}\\
0_{k-1}^\top &1
\end{pmatrix}=
\begin{pmatrix}
C_{k-1}^{-1}&
\begin{smallmatrix}
 0_{k-2}\\
 -1
\end{smallmatrix}
\\
0_{k-1}^\top &1
\end{pmatrix}
\end{equation}
where $0_i$ (respectively $1_i$) is the column of length $i$ composed of $0$'s (respectively of $1$'s) and, in the last equality, we use the relation
\begin{equation}
C_{k}^{-1}1_k=
\begin{pmatrix}
 0_{k-1}\\
 1
\end{pmatrix}.
\end{equation}

As is shown in~\cite{MR2435235}, for any $\beta\in B_n$, one has the the equality
\begin{equation}
C_n^{-1}\psi_n(\beta)C_n=
\begin{pmatrix}
 \psi_n^r(\beta)&0_{n-1}\\
 *_\beta &1
\end{pmatrix}
\end{equation}
where $\psi_n^r\colon B_n\to \operatorname{GL}_{n-1}(\Lambda)$ is the reduced Burau representation, and $*_\beta$ is a row of length $n-1$ over $\Lambda$ linearly depending on the rows $a_i$, $1\le i\le n-1$, of the matrix  $\psi_n^r(\beta)-I_{n-1}$ through the formula\footnote{This is the content of Lemma~3.10 of \cite{MR2435235} where the formula is written with a typo.}
\begin{equation}\label{eq:lin-dep-row}
(1-t^n)*_\beta=\sum_{i=1}^{n-1}(t^i-1) a_i.
\end{equation}
\begin{lemma}\label{lem-2}
 Let $\hat\beta_n$ be the $(n-1)\times(n-1)$ matrix obtained from $\psi_n(\beta)$ by throwing away the $n$-th row and the $n$-th column. Then, one has the following equality in $\Lambda$:
 \begin{equation}
(t^{-n}-1) \det(\hat\beta_n-I_{n-1})=(t^{-1}-1)\det(\psi_n^r(\beta)-I_{n-1}).
\end{equation}
\end{lemma}
\begin{proof}
We have the following equality of matrices:
\begin{multline}
\hat\beta_n=(C_{n-1}\psi_n^r(\beta)+1_{n-1}*_\beta)C_{n-1}^{-1}
\\
\Leftrightarrow\quad  C_{n-1}^{-1}\hat\beta_nC_{n-1}=\psi_n^r(\beta)+C_{n-1}^{-1}1_{n-1}*_\beta
=
\psi_n^r(\beta)+\left(\begin{smallmatrix}
 0_{n-2}\\
 1
\end{smallmatrix}\right)*_\beta.
%=\begin{pmatrix}a_1^\top&\ldots&a_{n-2}^\top&(a_{n-1}-*_\beta)^\top \end{pmatrix}^\top
\end{multline}
One can verify this by explicit calculation based on the block structure~\eqref{eq:block-c_n}:
\begin{multline}
 \psi_n(\beta)=
\begin{pmatrix}
C_{n-1}&1_{n-1}\\
0_{n-1}^\top&1
\end{pmatrix}
\begin{pmatrix}
 \psi_n^r(\beta)&0_{n-1}\\
 *_\beta &1
\end{pmatrix} C_{n}^{-1}
\\
=\begin{pmatrix}
C_{n-1}\psi_n^r(\beta)+1_{n-1}*_\beta&*\\
*_\beta &1
\end{pmatrix}
\begin{pmatrix}
C_{n-1}^{-1}&*\\
0_{n-1}^\top &1
\end{pmatrix}\\
=\begin{pmatrix}
(C_{n-1}\psi_n^r(\beta)+1_{n-1}*_\beta)C_{n-1}^{-1}&*\\
* &1
\end{pmatrix}.
\end{multline}
Thus,
\begin{equation}\label{eq:det-id-1}
\det(\hat\beta_n-I_{n-1})=\det\left(\psi_n^r(\beta)-I_{n-1}+\left(\begin{smallmatrix}
 0_{n-2}\\
 1
\end{smallmatrix}\right)*_\beta\right)
=\det 
\left(\begin{smallmatrix}
 a_1\\
 \vdots\\
 a_{n-2}\\
 a_{n-1}+*_\beta
\end{smallmatrix}
\right).
\end{equation}
By multiplying both sides of \eqref{eq:det-id-1}
by $(1-t^n)$ and using~\eqref{eq:lin-dep-row}, we obtain
\begin{multline*}
(1-t^n) \det(\hat\beta_n-I_{n-1})=\det\!\! 
\left(\begin{smallmatrix}
 a_1\\
 \vdots\\
 a_{n-2}\\
(1-t^n)( a_{n-1}+*_\beta)
\end{smallmatrix}
\right)
=\det\!\! 
\left(\begin{smallmatrix}
 a_1\\
 \vdots\\
 a_{n-2}\\
(1-t^n)a_{n-1}+\sum_{i=1}^{n-1}(t^i-1)a_i
\end{smallmatrix}
\right)\\
=\det\!\!
\left(\begin{smallmatrix}
 a_1\\
 \vdots\\
 a_{n-2}\\
(t^{-1}-1)t^na_{n-1}
\end{smallmatrix}
\right)
=(t^{-1}-1)t^n\det(\psi_n^r(\beta)-I_{n-1})
\end{multline*}
where in the third equality we dropped from the sum all the terms proportional to the rows different from $n-1$.
\end{proof}
\begin{proof}[Proof of Theorem~\ref{thm-2}]
 The formula for the Alexander polynomial proven in~\cite[Theorem~3.13]{MR2435235} is of the form
 \begin{equation}
\Delta_K(t)=(-1)^{n-1}t^{(n-1-g(\beta))/2}\frac{t-1}{t^n-1}\det(\psi_n^r(\beta)-I_{n-1})
\end{equation}
which is equivalent to \eqref{eq:re-alex-det-bur} due to Lemma~\ref{lem-2}.
\end{proof}
%\bibliography{/Users/rinatkashaev/Dropbox/LatexStaff/biblio}{}

\begin{thebibliography}{10}

\bibitem{MR1389962}
D.~Bar-Natan and S.~Garoufalidis.
\newblock On the {M}elvin-{M}orton-{R}ozansky conjecture.
\newblock {\em Invent. Math.}, 125(1):103--133, 1996.

\bibitem{MR0375281}
J.~S. Birman.
\newblock {\em Braids, links, and mapping class groups}.
\newblock Princeton University Press, Princeton, N.J.; University of Tokyo
  Press, Tokyo, 1974.
\newblock Annals of Mathematics Studies, No. 82.

\bibitem{MR2186115}
A.~Brugui{\`e}res and A.~Virelizier.
\newblock Hopf diagrams and quantum invariants.
\newblock {\em Algebr. Geom. Topol.}, 5:1677--1710 (electronic), 2005.

\bibitem{MR3069587}
W.~Burau.
\newblock Kennzeichnung der {S}chlauchknoten.
\newblock {\em Abh. Math. Sem. Univ. Hamburg}, 9(1):125--133, 1933.

\bibitem{MR3069631}
W.~Burau.
\newblock Kennzeichnung der schlauchverkettungen.
\newblock {\em Abh. Math. Sem. Univ. Hamburg}, 10(1):285--297, 1934.

\bibitem{MR3069652}
W.~Burau.
\newblock \"{U}ber {Z}opfgruppen und gleichsinnig verdrillte {V}erkettungen.
\newblock {\em Abh. Math. Sem. Univ. Hamburg}, 11(1):179--186, 1935.

\bibitem{MR1786197}
S.~D{\u{a}}sc{\u{a}}lescu, C.~N{\u{a}}st{\u{a}}sescu, and {\c{S}}.~Raianu.
\newblock {\em Hopf algebras}, volume 235 of {\em Monographs and Textbooks in
  Pure and Applied Mathematics}.
\newblock Marcel Dekker, Inc., New York, 2001.
\newblock An introduction.

\bibitem{MR2860990}
S.~Garoufalidis and T.~T.~Q. L\^{e}.
\newblock Asymptotics of the colored {J}ones function of a knot.
\newblock {\em Geom. Topol.}, 15(4):2135--2180, 2011.

\bibitem{MR2253443}
K.~Habiro.
\newblock Bottom tangles and universal invariants.
\newblock {\em Algebr. Geom. Topol.}, 6:1113--1214, 2006.

\bibitem{Kashaev2019}
R.~Kashaev.
\newblock Invariants of long knots.
\newblock arXiv:1908.00118, 2019.

\bibitem{MR2796628}
R.~M. Kashaev.
\newblock {$R$}-matrix knot invariants and triangulations.
\newblock In {\em Interactions between hyperbolic geometry, quantum topology
  and number theory}, volume 541 of {\em Contemp. Math.}, pages 69--81. Amer.
  Math. Soc., Providence, RI, 2011.

\bibitem{MR2435235}
C.~Kassel and V.~Turaev.
\newblock {\em Braid groups}, volume 247 of {\em Graduate Texts in
  Mathematics}.
\newblock Springer, New York, 2008.
\newblock With the graphical assistance of Olivier Dodane.

\bibitem{MR1133269}
L.~H. Kauffman and H.~Saleur.
\newblock Free fermions and the {A}lexander-{C}onway polynomial.
\newblock {\em Comm. Math. Phys.}, 141(2):293--327, 1991.

\bibitem{MR1124415}
R.~J. Lawrence.
\newblock A universal link invariant using quantum groups.
\newblock In {\em Differential geometric methods in theoretical physics
  ({C}hester, 1988)}, pages 55--63. World Sci. Publ., Teaneck, NJ, 1989.

\bibitem{MR1153694}
H.~C. Lee.
\newblock Tangles, links and twisted quantum groups.
\newblock In {\em Physics, geometry, and topology ({B}anff, {AB}, 1989)},
  volume 238 of {\em NATO Adv. Sci. Inst. Ser. B Phys.}, pages 623--655.
  Plenum, New York, 1990.

\bibitem{MR1324033}
V.~Lyubashenko.
\newblock Tangles and {H}opf algebras in braided categories.
\newblock {\em J. Pure Appl. Algebra}, 98(3):245--278, 1995.

\bibitem{MR0141605}
P.~A. MacMahon.
\newblock {\em Combinatory analysis}.
\newblock Two volumes (bound as one). Chelsea Publishing Co., New York, 1960.

\bibitem{MR1381692}
S.~Majid.
\newblock {\em Foundations of quantum group theory}.
\newblock Cambridge University Press, Cambridge, 1995.

\bibitem{MR2466562}
J.~Murakami and K.~Nagatomo.
\newblock Logarithmic knot invariants arising from restricted quantum groups.
\newblock {\em Internat. J. Math.}, 19(10):1203--1213, 2008.

\bibitem{MR1227011}
T.~Ohtsuki.
\newblock Colored ribbon {H}opf algebras and universal invariants of framed
  links.
\newblock {\em J. Knot Theory Ramifications}, 2(2):211--232, 1993.

\bibitem{MR858831}
A.~Perelomov.
\newblock {\em Generalized coherent states and their applications}.
\newblock Texts and Monographs in Physics. Springer-Verlag, Berlin, 1986.

\bibitem{MR1025161}
N.~Y. Reshetikhin.
\newblock Quasitriangular {H}opf algebras and invariants of links.
\newblock {\em Algebra i Analiz}, 1(2):169--188, 1989.

\bibitem{MR1036112}
N.~Y. Reshetikhin and V.~G. Turaev.
\newblock Ribbon graphs and their invariants derived from quantum groups.
\newblock {\em Comm. Math. Phys.}, 127(1):1--26, 1990.

\bibitem{MR1612375}
L.~Rozansky.
\newblock The universal {$R$}-matrix, {B}urau representation, and the
  {M}elvin-{M}orton expansion of the colored {J}ones polynomial.
\newblock {\em Adv. Math.}, 134(1):1--31, 1998.

\bibitem{Salter2019}
N.~Salter.
\newblock Linear-central filtrations and the image of the {B}urau
  representation.
\newblock arXiv:1903.11209, 2019.

\bibitem{MR2251160}
A.~Virelizier.
\newblock Kirby elements and quantum invariants.
\newblock {\em Proc. London Math. Soc. (3)}, 93(2):474--514, 2006.

\bibitem{MR547117}
W.~C. Waterhouse.
\newblock {\em Introduction to affine group schemes}, volume~66 of {\em
  Graduate Texts in Mathematics}.
\newblock Springer-Verlag, New York-Berlin, 1979.

\end{thebibliography}
%\bibliographystyle{abbrv}
\def\cprime{$'$} \def\cprime{$'$}

  \end{document}